%% file: particles.tex
\documentclass[12pt]{amsart}
\usepackage{amsmath,amsthm}
\usepackage{graphicx}
   \graphicspath{{figures/}}
\usepackage[usenames,dvipsnames]{color}
\numberwithin{equation}{section}

\newtheorem{cor}[equation]{Corollary}
\newtheorem{lm}[equation]{Lemma}

\theoremstyle{definition}

\theoremstyle{remark}

\begin{document}
\date{\today}
\title{R\'{e}nyi's Parking Problem Revisited}
\author{Matthew P. Clay and Nandor J. Simanyi}
\address[Matthew P. Clay]{Georgia Institute of Technology\\
School of Aerospace Engineering\\
270 Ferst Drive Atlanta GA 30332}
\address[N\'andor Sim\'anyi]{The University of Alabama at Birmingham\\
Department of Mathematics\\
1300 University Blvd., Suite 452\\
Birmingham, AL 35294}

\email[Matthew P. Clay]{mclay6@gatech.edu}
\email[Nandor J. Simanyi]{simanyi@uab.edu}

\begin{abstract}
R\'enyi's parking problem (or $1D$ sequential interval packing problem) dates back to 1958, when R\'enyi studied the following 
random process: Consider an interval $I$ of length $x$, and sequentially and randomly pack disjoint unit intervals in $I$ until 
the remaining space prevents placing any new segment. The expected value of the measure of the covered part of $I$ is $M(x)$, so 
that the ratio $M(x)/x$ is the expected filling density of the random process. Following recent work by Gargano 
{\it et al.} \cite{GWML(2005)}, we studied the discretized version of the above process by considering the 
packing of the $1D$ discrete lattice interval $\{1,2,\dots ,n+2k-1\}$ with disjoint blocks of $(k+1)$ integers but, as opposed to the 
mentioned \cite{GWML(2005)} result, our exclusion process is symmetric, hence more natural. Furthermore, we were able to 
obtain useful recursion formulas for the expected number of $r$-gaps ($0\le r\le k$) between neighboring blocks. We also 
provided very fast converging series and extensive computer simulations for these expected numbers, so that the limiting 
filling density of the long line segment (as $n\to \infty$) is R\'enyi's famous parking constant, $0.7475979203\dots$.
\end{abstract}

\maketitle

\parskip=0pt plus 2.5pt

\bigskip \bigskip

\section{Introduction}

\bigskip \bigskip

R\'enyi's Parking Problem (or $1D$ sequential interval packing problem) dates back to 1958 when R\'enyi \cite{R(1958)}
studied the probabilistic properties of the following random process: Consider an interval $I$ of length $x>>1$ ($x$ will 
eventually tend to infinity), and sequentially and randomly pack disjoint unit intervals in $I$ 
as long as the remaining space permits placing any new unit segment in $I$. At each step of the packing process the position of the newly placed interval 
is chosen uniformly from the available space. Denote the expected value of the measure of the covered part by $M(x)$, so 
that the ratio $M(x)/x$ is the expected filling density of the ``parking process". (The interval $I$ is the street curb, and 
the packed unit segments are the parked cars.) R\'enyi himself proves the following continuous recursion for $M(x)$
\begin{equation}\label{integral-recursion}
M(x)=\begin{cases}
0 \text{ for } 0\le x<1 \\
1+\frac{2}{x-1}\int_0^{x-1} M(y)dy \text{ for } x\ge 1,
\end{cases}
\end{equation}
and from this he deduces the asymptotic mean filling density
\begin{equation}\label{renyi-constant}
m=\lim_{x\to\infty} M(x)/x=\int_0^\infty \exp\left[-2\int_0^x \frac{1-e^{-y}}{y}dy\right]dx
=0.7475979203\dots ,
\end{equation}
which number is now known as R\'enyi's Parking Constant. R\'enyi \cite{R(1958)} further proves the asymptotic formula
\begin{equation}\label{asymptotics-1}
M(x)=mx+m-1+\mathcal{O}(x^{-n})
\end{equation}
for every positive integer $n$, which was further improved by Dvoretzky and Robbins \cite{DR(1964)} to
\begin{equation}\label{asymptotics-2}
M(x)=mx+m-1+\mathcal{O}\left[\left(\frac{2e}{x}\right)^{x-3/2}\right].
\end{equation}
In that paper Dvoretzky and Robbins also prove that
\[
\inf_{x\le t\le x+1}\frac{M(t)+1}{t+1}\le m\le 
\sup_{x\le t\le x+1}\frac{M(t)+1}{t+1}.
\]

The first ``discretized" version of the problem, namely the expected density derived from sequential packings of 
non-overlapping neighboring pairs of integer points, i.e., edges or bonds, selected at random on a long segment of a $1D$ 
lattice was first given by Page \cite{P(1959)}. His results have been confirmed and extended in various ways by Downton
\cite{D(1961)}, Mackenzie \cite{M(1962)}, Widom \cite{W(1966)}, and Solomon \cite{S(1967)}. This random sequential addition 
model is pertinent when molecules are sequentially absorbed, and once absorbed, are fixed. The expected density derived by
non-sequentially packing disjoint, indistinguishable, neighboring pairs of points on a linear lattice, each configuration 
being considered equally likely, was first given by Jackson and Montroll \cite{JM(1958)}, and extended to neighboring 
triplets by Fisher and Temperly \cite{FT(1960)}.

Following a more recent paper by Gargano {\it et al.} \cite{GWML(2005)} 
we studied the discretized version 
of the above process by considering the sequential packing of the $1D$ discrete lattice interval
$\{1,2,\dots,n+2k-1\}$ ($n>>1$) with disjoint blocks of $k+1$ consecutive integers but, as opposed to the approach in
\cite{GWML(2005)}, our packing process is symmetric, hence more natural. Furthermore, we were able to obtain useful 
recursions for the expected number of $r$-gaps ($0\le r\le k$) between neighboring blocks (cars). The construction of such a 
recursion is one of the open problems listed at the end of \cite{GWML(2005)}.

We also provided very fast (faster than any exponential) converging series for the expected number of $r$-gaps, and carried 
out extensive computer simulations for these expected numbers, indicating that the limiting filling density is indeed 
R\'enyi's famous parking constant $m=0.7475979203\dots$, also in the discrete parking problem.

It has to be noted, however, that our approach differs slightly, albeit just in minor technical terms, from the approach in
\cite{GWML(2005)}. Namely, instead of considering packings with disjoint $(k+1)$-blocks of consecutive integer lattice 
points, i.e., with $k$ consecutive edges or bonds between them, we consider the positions of the centers of these blocks. 
The available space for the centers is either the original integer lattice (when $k$ is even) or the original lattice 
shifted by $1/2$ units (when $k$ is odd), so that the distance between neighboring centers is always at least $k+1$, 
i.e., the gap between them contains at least $k$ points. This approach is clearly equivalent to that of \cite{GWML(2005)}.

Finally, in the paper \cite{GWML(2005)} the authors consider the process in which not only the distances between neighboring 
centers of blocks are at least $k+1$ but, in addition, each center is distanced at least $k+2$ from at least one of its 
two neighbors (an asymmetric model). At the end of their paper among the open problems they list the need to study the 
symmetric model in which we drop the second lower bound requirement ($\ge k+2$). This is exactly the kind of model we are 
investigating in this paper.

\bigskip \bigskip

\section{The Model}

\bigskip \bigskip

For incoming cars (i.e. centers of $(k+1)$-blocks of consecutive
integers as described in the introduction) there are $n+k-1$ parking
slots, labelled as $1,2,\dots,n+k-1$ ($n, k\ge 1$), in a row that the
cars, arriving one-by-one, want to occupy. The drivers have the desire
that the distance between occupied parking slots is at least $k+1$,
i.e., the gap between neighboring cars (and also the gap before the
first car and after the last one) contains at least $k$ unoccupied
slots. ($k$ being a fixed positive integer.)
When a new car arrives, the driver considers all available slots and occupies one of them with equal probability.
The process lasts as long as the cars can occupy parking slots.

At the end of the process there will be gaps of sizes $k, k+1, \dots,2k$. For any $r$, $k\le r\le 2k$, and for any positive
integer $n$ let $a_n^{(r)}$ be the expected number of $r$-gaps produced by the above random process.

Since the events $A_i$ ($1\le i\le n-k-1$) that the first arriving car occupies the slot $i+k$ are equally probable,
pairwise exclusive, and their union is the sure event, one immediately gets the recursion formula
\begin{equation}
a_n^{(r)}=\frac{2}{n-k-1}\sum_{i=1}^{n-k-1}a_i^{(r)}
\end{equation}
for $n\ge k+2$. Furthermore, the initial conditions
\begin{equation}
a_n^{(r)}=
\begin{cases} 
1 \text{ if } n=r-k+1 \\
0 \text{ if } 1\le n\le k+1,\; n\ne r-k+1
\end{cases}
\end{equation}
hold true, $k\le r\le 2k$.

We take $s_n^{(r)}=\sum_{i=1}^n a_i^{(r)}$, $t_n^{(r)}=\dfrac{s_n^{(r)}}{n(n+2k+1)}$, so that
\[
s_n^{(r)}=s_{n-1}^{(r)}+\frac{2}{n-k-1}\cdot s_{n-k-1}^{(r)}
\]
and
\begin{equation}\label{t-recursion}
n(n+2k+1)t_n^{(r)}=(n-1)(n+2k)t_{n-1}^{(r)}+2(n+k)t_{n-k-1}^{(r)}
\end{equation}
for $n\ge k+2$, $k\le r\le 2k$. 

For $n\ge 2$ define $u_n^{(r)}=t_n^{(r)}-t_{n-1}^{(r)}$. From (\ref{t-recursion}) elementary calculation yields
the $k$-step linear recursion
\begin{equation}\label{big-recursion}
u_n^{(r)}=\frac{-2(n+k)}{n(n+2k+1)}\cdot \sum_{i=1}^k u_{n-i}^{(r)}
\end{equation}
for $n\ge k+2$.
The initial values $\left\{u_n^{(r)}\big|\; 2\le n\le k+1\right\}$
for the $\left(u_n^{(r)}\right)_{n=2}^\infty$ sequence are as follows:

\medskip

\begin{equation}\label{u-initial}
u_n^{(r)}=
\begin{cases}
0 &\text{ if } 2\le n\le r-k \\
\frac{1}{(r-k+1)(r+k+2)} &\text{ if }n=r-k+1\text{ and }r\ge k+1 \\
\frac{1}{n(n+2k+1)}-\frac{1}{(n-1)(n+2k)} &\text{ if } r-k+2\le n\le k+1.
\end{cases}
\end{equation}

\medskip \medskip \medskip

The following formulas are immediate consequences of the definitions of the involved quantities.

\medskip

\begin{equation}\label{t-formula}
t_n^{(r)}=\frac{s_1^{(r)}}{2k+2}+\sum_{i=2}^n u_i^{(r)}, \;\; n\ge 1,
\end{equation}
where
\[
s_1^{(r)}=
\begin{cases}
1 &\text{ if } r=k \\
0 &\text{ if } k<r\le 2k.
\end{cases}
\]

\medskip

\begin{equation}\label{s-formula}
s_n^{(r)}=n(n+2k+1)t_n^{(r)}, \;\; n\ge 1,
\end{equation}

\medskip

\begin{equation}\label{a-formula}
a_n^{(r)}=s_n^{(r)}-s_{n-1}^{(r)}, \;\; n\ge 1,
\end{equation}
where $s_0^{(r)}=0$ by convention.

\bigskip

\begin{cor}
The limiting densities
\begin{equation}
D(k,r)=(r+1)\lim_{n\to\infty}\frac{a_n^{(r)}}{n}=2(r+1)t_\infty^{(r)}
\end{equation}
exist for all $r$, $k\le r\le 2k$. Clearly
\begin{equation}\label{normalization}
\sum_{r=k}^{2k} D(k,r)=1.
\end{equation}
\end{cor}

\begin{proof}
According to the previous corollary $t_n^{(r)}=t_\infty^{(r)}+\mathcal{O}(a^n)$ with an arbitrarily small constant $a>0$.
Therefore
\[
s_n^{(r)}=n(n+2k+1)[t_\infty^{(r)}+\mathcal{O}(a^n)]=n(n+2k+1)t_\infty^{(r)}+\mathcal{O}(a^n),
\]
and
\[
a_n^{(r)}=s_n^{(r)}-s_{n-1}^{(r)}=2(n+k)t_\infty^{(r)}+\mathcal{O}(a^n).
\]
\end{proof}

\bigskip

\section{Fundamental calculations}

\bigskip \bigskip

From \ref{big-recursion} for $n\ge k+2$ one gets
\[
|u_n^{(r)}|\le\frac{2(n+k)}{n(n+2k+1)}\cdot\sum_{i=1}^k |u_{n-i}^{(r)}| \\
\le \frac{2k}{n}\left(\frac{1}{k}\sum_{i=1}^k |u_{n-i}^{(r)}|\right),
\]
so for the non-negative numbers $w_n^{(r)}=|u_n^{(r)}|$ ($n\ge 2$) one obtains
\begin{equation}\label{w-recursive-estimate}
w_n^{(r)}\le\frac{2k}{n}\left(\frac{1}{k}\sum_{i=1}^k w_{n-i}^{(r)}\right)
\end{equation}
for $n\ge k+2$.

\medskip

\begin{lm}\label{w-upper-bound}
For $n\ge 2$ write $n=pk+s$ with $p\ge 0$ and $2\le s\le k+1$. We claim that the inequality 
\begin{equation}\label{w-bound}
|u_n^{(r)}|=w_n^{(r)}\le \frac{M\cdot 2^p}{p!}
\end{equation}
holds true with the constant 
\[
M=M_{k,r}=\max\left\{w_n^{(r)}\big|\; 2\le n\le 2k\right\},
\]
depending only on $k$ and $r$.
\end{lm}

\medskip

\begin{proof}
We define the auxiliary sequence $\left(w_n^{(r)}\right)'=w_n'$ ($n\ge 2$) with the following recursion:
\begin{equation}\label{w-prime-recursion}
w_n'=
\begin{cases}
M\text{ for } 2\le n\le 2k \\
\frac{2}{n}\cdot\sum_{i=1}^k w'_{n-i} \text{ for } n\ge 2k+1.
\end{cases}
\end{equation}
Since the expression on the right-hand-side of (\ref{w-recursive-estimate}) is monotone increasing in its variables
$w_{n-i}^{(r)}$ (the coefficients being positive), we immediately get the bounds
\begin{equation}\label{w-compare}
w_n^{(r)}\le w_n'
\end{equation}
for $n\ge 2$. Also, it is clear from the recursion (\ref{w-prime-recursion}) that 
$M>w'_{2k+2}>w'_{2k+3}>w'_{2k+4}>\dots$, so
\begin{equation}\label{k-step-estimate}
w'_n\le \frac{2k}{n}\cdot w'_{n-k}
\end{equation}
for $n\ge k+2$. By an obvious induction, the inequalities (\ref{k-step-estimate}) above yield
\begin{equation}
w'_n\le\frac{M(2k)^p}{n(n-k)(n-2k)\dots(n-pk+k)}\le\frac{M(2k)^p}{k^p p!}
=\frac{M\cdot 2^p}{p!},
\end{equation}
where $n=pk+s$ with $p\ge 0$ and $2\le s\le k+1$.
\end{proof}

\medskip

\begin{cor}\label{fast-u-convergence}
For the $n$-th error term 
\[
R_n=\sum_{i\ge n} u_i^{(r)}=t_\infty^{(r)}-t_{n-1}^{(r)}
\]
of the absolutely convergent series 
\begin{equation}\label{t-series}
t_\infty^{(r)}=\lim_{i\to \infty} t_i^{(r)}=\frac{s_1^{(r)}}{2k+2}+\sum_{i=2}^\infty u_i^{(r)}
\end{equation}
we have the superexponential upper bound
\begin{equation}\label{strong-bound}
|R_n|\le\frac{Mke^2\cdot 2^{p_n}}{p_n!},
\end{equation}
where $n=p_nk+s$ with $2\le s\le k+1$.
\end{cor}

\medskip

\begin{proof}
\begin{equation}
\begin{aligned}
&\left|\sum_{i\ge n} u_i^{(r)}\right|\le\sum_{i\ge n} |u_i^{(r)}|\le M\cdot\sum_{i\ge n} \frac{2^{p_i}}{p_i!} \\
&\le Mk\cdot\sum_{p\ge p_n}\frac{2^{p}}{p!}\le \frac{Mke^2\cdot 2^{p_n}}{p_n!},
\end{aligned}
\end{equation}
according to the usual upper bound for the $p_n$-th error term of the Taylor
expansion of the exponential function.
\end{proof}

\bigskip \bigskip

\section{Behavior of the Limiting Densities $D(k,r)$}

\bigskip \bigskip

We conjecture that, for a given $k$, $D(k,r)$ is decreasing in $r$,
and $kD(k,2k)>0$ is separated from $0$, uniformly in $k$.  Please keep
in mind \ref{normalization}, indicating that the proper normalization
(to get non-zero limit) of the densities $D(k,r)$ is $kD(k,r)$. As
follows, we present strong numerical evidence for this.  Such
numerical evidence is certainly feasible for, according to Corollary
\ref{fast-u-convergence}, the partial sums of the series
(\ref{t-series}) converge faster than any exponential function,
therefore all the formulas (\ref{big-recursion}), (\ref{t-formula}),
(\ref{s-formula}), and (\ref{a-formula}) converge very fast with error
terms that are easy to effectively estimate.


Of particular interest is the limiting cumulative distribution function

\begin{equation}\label{cumulative}
F(t)=\lim_{k\to\infty} \sum_{r=k}^{[(1+t)k]} D(k,r),
\end{equation}
and the corresponding limiting density function

\begin{equation}\label{density-function}
F'(t)=\lim_{k\to\infty} kD(k,[(1+t)k])
\end{equation}
for $0\le t\le 1$. 


To support these conjectures, calculations were conducted for a wide range of
$k$ values, and the results for $k = 2^{20}$ are presented in Figures 
\ref{fig:Dkr}, \ref{fig:F}, and \ref{fig:dFdt}.

\pagebreak
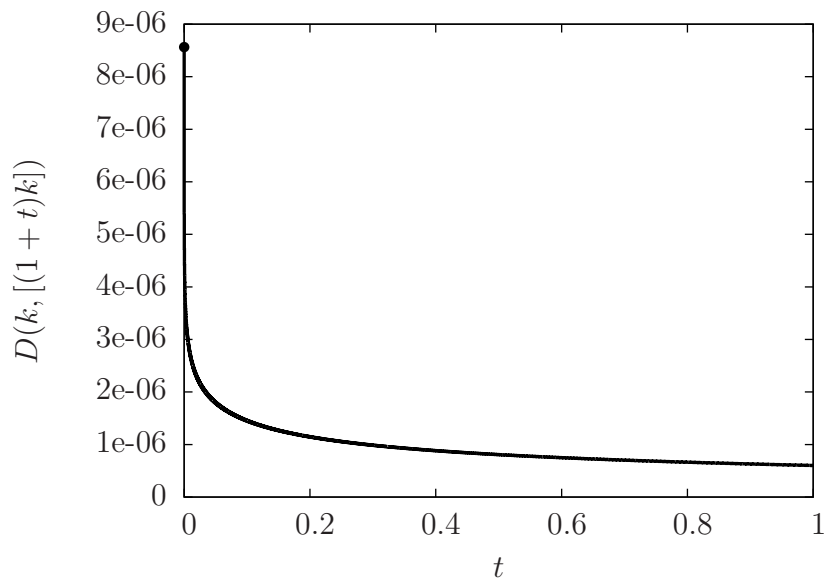
\begin{figure}[h!]
   \begin{center}
      \input{figures/Dkr.tex}
   \end{center}
   \vspace{-12pt}
   \caption{Plot of the values of $D(k,r)$ for $k=2^{20}$ versus a normalized
   axis $t = (r - k)/k$. The maximum value obtained is denoted by the symbol
   at $t=0$.}
   \label{fig:Dkr}
\end{figure}
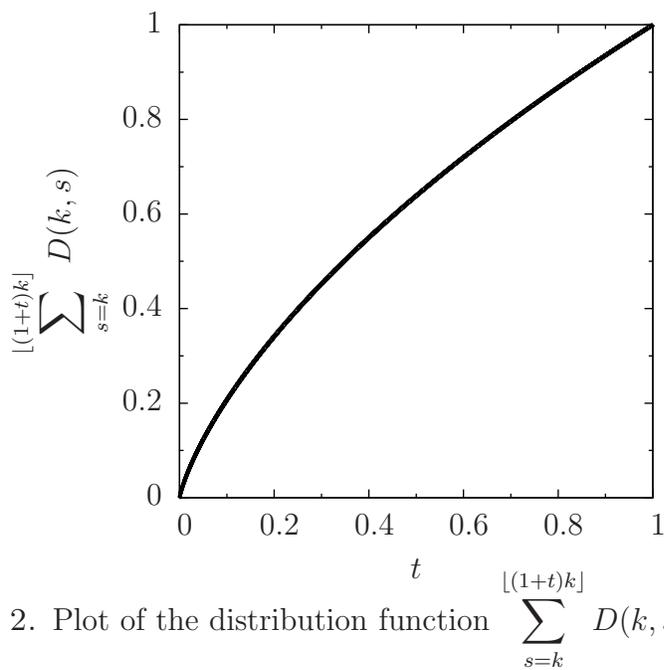
\begin{figure}[h!]
   \begin{center}
      \input{figures/F.tex}
   \end{center}
   \vspace{-24pt}
   \caption{Plot of the distribution function 
            $\displaystyle \sum_{s=k}^{\lfloor (1+t)k \rfloor} D(k,s)$ for $k=2^{20}$.}
   \label{fig:F}
\end{figure}

\pagebreak

\begin{figure}[h!]
   \begin{center}
      \input{figures/dFdt.tex}
   \end{center}
   \vspace{-12pt}
   \caption{Plot of the density function $kD(k,[(1+t)k])$ for $k=2^{20}$.
            The maximum value is at $t=0$ and is marked with the symbol.}
   \label{fig:dFdt}
\end{figure}
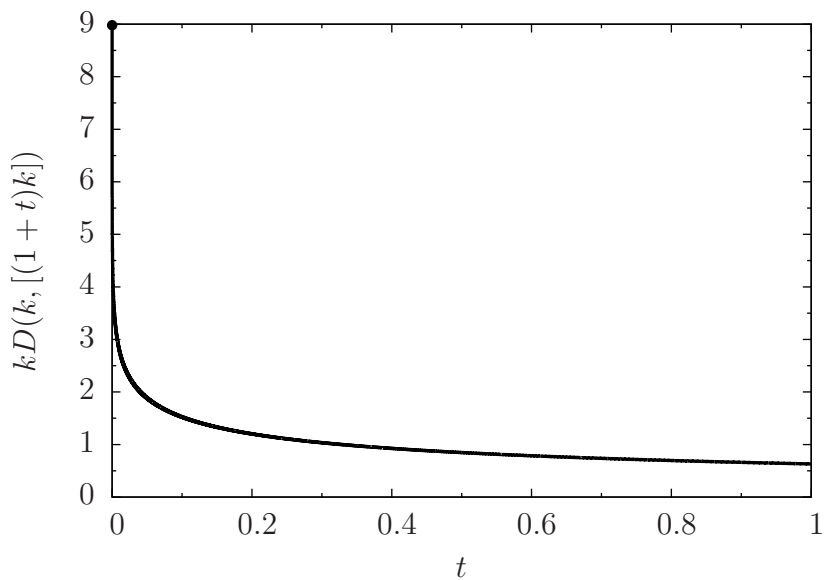

\noindent The density function $kD(k,[(1+t)k])$ exhibits interesting 
behavior at $t = 0$.
To demonstrate this, Figure \ref{fig:der0} presents data for $kD(k,k)$
spanning many decades of $k$.
\begin{figure}[h!]
   \begin{center}
      \input{figures/kDkk.tex}
   \end{center}
   \vspace{-12pt}
   \caption{Plot of the growth of $kD(k,k)$ as $k$ is increased. The
   values of $k$ used were $2^n$, where $3\le n \le 30$.}
   \label{fig:der0}
\end{figure}
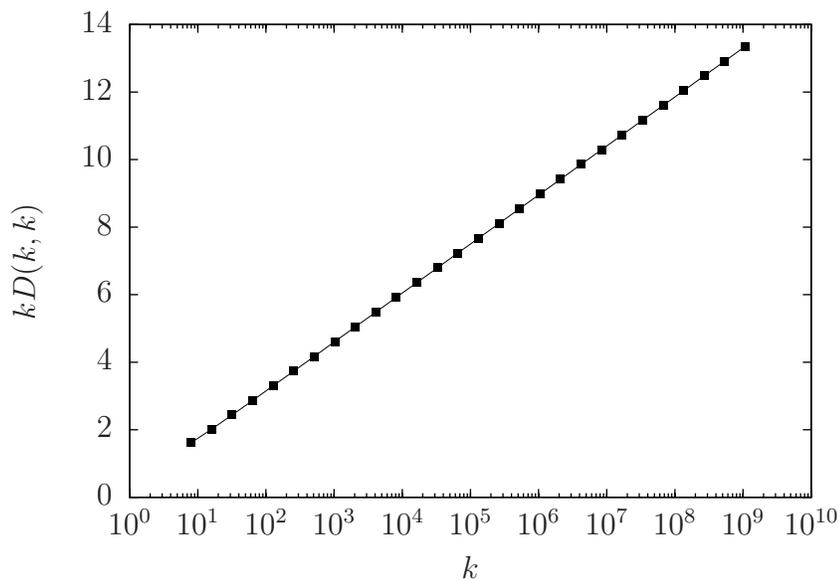

\noindent The results given in Figure \ref{fig:der0} show that $kD(k,k)$
grows at a logarithmic rate with $k$.
Similarly, we investigate the behavior of $kD(k,2k)$ in Figure \ref{fig:der1}.
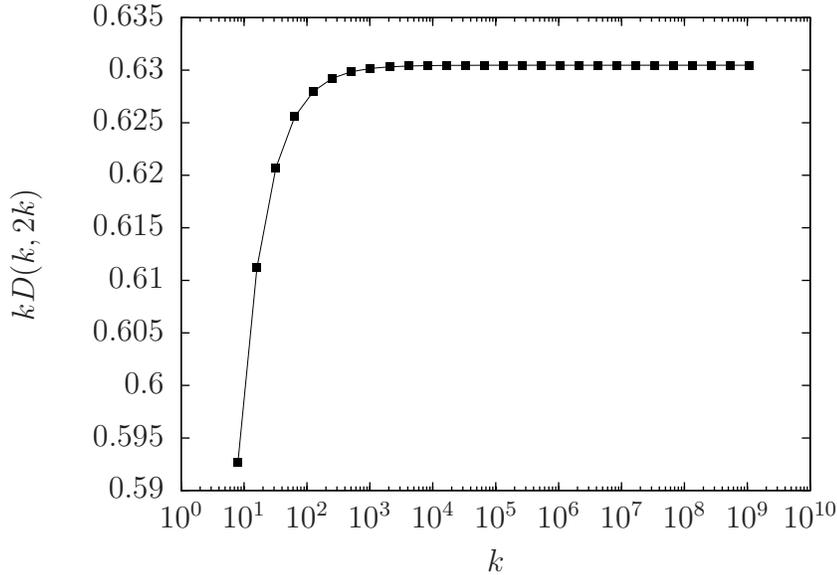
\begin{figure}[h!]
   \begin{center}
      \input{figures/kDk2k.tex}
   \end{center}
   \vspace{-12pt}
   \caption{Plot of $kD(k,2k)$ as $k$ is increased. The values of $k$ used 
            were $2^n$, where $3\le n \le 30$.}
   \label{fig:der1}
\end{figure}

\noindent The results in Figure \ref{fig:der1} support the conjecture that the
sequence $\{kD(k,2k)\}$ converges to a number $0.6304735\dots$ in a monotone 
increasing fashion as $k\to\infty$.
Figures \ref{fig:dFdt}, \ref{fig:der0}, and \ref{fig:der1} serve as
strong numerical evidence for the claim that the limiting density function 
$F'(t)$ of (\ref{density-function}) continuously decreases from infinity
at $t=0$ to a positive constant $0.6304735\dots$. Furthermore, it is worth noting that
these pictures are in pretty good harmony with the results of \S4 of \cite{Man(1976)} 
on the distribution of the gap lengths.


The number $D(k)=\sum_{r=k}^{2k} \dfrac{k+1}{r+1}D(k,r)$ has a special meaning: It is the limiting filling density of cars (i. e. $(k+1)$-blocks) getting a parking slot,
as $n\to\infty$. Clearly $\dfrac{k+1}{2k+1}\le D(k)\le 1$.

\medskip

Particularly interesting is the limiting packing density
\[
D=\lim_{k\to\infty} D(k).
\]
Clearly $1/2\le D\le 1$.
The behavior of $D$ was investigated numerically and the results are presented
in Figure \ref{fig:D}. The obtained numerical evidence supports 
the claim that $D=m=0.7475979203\dots$ is R\'enyi's famous parking constant.
\pagebreak
\begin{figure}[h!]
   \begin{center}
      \input{figures/D.tex}
   \end{center}
   \caption{Plot depicting the difference between the calculated values of 
            $D(k)$ and R\'{e}nyi's constant $m$ (to machine precision) versus 
            $k$. The values of $k$ used were $2^n$, where $3\le n \le 20$.}
   \label{fig:D}
\end{figure}
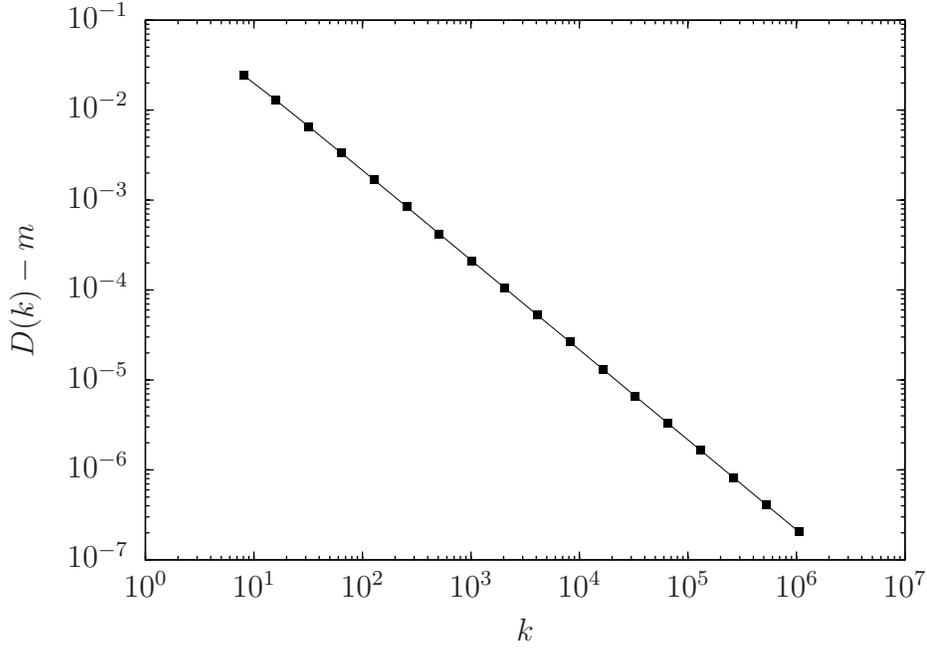

\subsection{Example. Detailed Calculations for $k=1$.}

If we take the case $k=r=1$, (\ref{big-recursion}) and (\ref{u-initial}) yield 
\[
u_2^{(1)}=-\dfrac{3}{20}, \quad u_n^{(1)}=\dfrac{-2(n+1)}{n(n+3)}u_{n-1}^{(1)}
\]
for $n\ge 3$, thus $u_n^{(1)}=\dfrac{3(n+1)(-2)^{n-1}}{(n+3)!}$
for $n\ge 2$, so
\begin{equation}
\begin{aligned}
& D(1,1)=4t_\infty^{(1)}=1+12\cdot\sum_{n=2}^\infty \frac{(n+3-2)(-2)^{n-1}}{(n+3)!} \\
& =1+12\cdot\sum_{n=2}^\infty \frac{(-2)^{n-1}}{(n+2)!}+12\cdot\sum_{n=2}^\infty \frac{(-2)^{n}}{(n+3)!} \\
& =1-\frac{3}{2}\left(\sum_{n=2}^\infty \frac{(-2)^{n+2}}{(n+2)!}+\sum_{n=2}^\infty \frac{(-2)^{n+3}}{(n+3)!}\right)
=1-3e^{-2}
\end{aligned}
\end{equation}
by obvious analysis technique.

Consequently, $D(1,2)=1-D(1,1)=3e^{-2}$, according to (\ref{normalization}). Finally, the exact value of the
filling density $D(1)$ is $D(1)=D(1,1)+\dfrac{2}{3}D(1,2)=1-e^{-2}$.
The filling density $D(1)=1-e^{-2}$ is in agreement with the filling density
$(1-e^{-2})/2$ obtained by Page \cite{P(1959)}.
The factor 2 discrepancy is due to the fact that we count each car with weight
$k+1=2$ (the space each car occupies), whereas Page uses weight 1.

\bigskip

\noindent {\it Remark}.
Whoever is interested in repeating the computer calculations, re-generating the
numerical plots, or checking the details in the source code of our programs, can
directly send us an e-mail message.
We will be more than happy to provide the code.

\bigskip \bigskip

\section*{Acknowledgements}

The authors express their sincere gratitude to Dr. Ya. Pesin for generously
supporting the conference participation of the first author via his NSF
Conference Grant.
Also, thanks to Steven Finch (Harvard University) for his helpful remarks.

\end{document}

%% file: figures/Dkr.tex
\begingroup
  \makeatletter
  \providecommand\color[2][]{%
    \GenericError{(gnuplot) \space\space\space\@spaces}{%
      Package color not loaded in conjunction with
      terminal option `colourtext'%
    }{See the gnuplot documentation for explanation.%
    }{Either use 'blacktext' in gnuplot or load the package
      color.sty in LaTeX.}%
    \renewcommand\color[2][]{}%
  }%
  \providecommand\includegraphics[2][]{%
    \GenericError{(gnuplot) \space\space\space\@spaces}{%
      Package graphicx or graphics not loaded%
    }{See the gnuplot documentation for explanation.%
    }{The gnuplot epslatex terminal needs graphicx.sty or graphics.sty.}%
    \renewcommand\includegraphics[2][]{}%
  }%
  \providecommand\rotatebox[2]{#2}%
  \@ifundefined{ifGPcolor}{%
    \newif\ifGPcolor
    \GPcolortrue
  }{}%
  \@ifundefined{ifGPblacktext}{%
    \newif\ifGPblacktext
    \GPblacktexttrue
  }{}%
  \let\gplgaddtomacro\g@addto@macro
  \gdef\gplbacktext{}%
  \gdef\gplfronttext{}%
  \makeatother
  \ifGPblacktext
    \def\colorrgb#1{}%
    \def\colorgray#1{}%
  \else
    \ifGPcolor
      \def\colorrgb#1{\color[rgb]{#1}}%
      \def\colorgray#1{\color[gray]{#1}}%
      \expandafter\def\csname LTw\endcsname{\color{white}}%
      \expandafter\def\csname LTb\endcsname{\color{black}}%
      \expandafter\def\csname LTa\endcsname{\color{black}}%
      \expandafter\def\csname LT0\endcsname{\color[rgb]{1,0,0}}%
      \expandafter\def\csname LT1\endcsname{\color[rgb]{0,1,0}}%
      \expandafter\def\csname LT2\endcsname{\color[rgb]{0,0,1}}%
      \expandafter\def\csname LT3\endcsname{\color[rgb]{1,0,1}}%
      \expandafter\def\csname LT4\endcsname{\color[rgb]{0,1,1}}%
      \expandafter\def\csname LT5\endcsname{\color[rgb]{1,1,0}}%
      \expandafter\def\csname LT6\endcsname{\color[rgb]{0,0,0}}%
      \expandafter\def\csname LT7\endcsname{\color[rgb]{1,0.3,0}}%
      \expandafter\def\csname LT8\endcsname{\color[rgb]{0.5,0.5,0.5}}%
    \else
      \def\colorrgb#1{\color{black}}%
      \def\colorgray#1{\color[gray]{#1}}%
      \expandafter\def\csname LTw\endcsname{\color{white}}%
      \expandafter\def\csname LTb\endcsname{\color{black}}%
      \expandafter\def\csname LTa\endcsname{\color{black}}%
      \expandafter\def\csname LT0\endcsname{\color{black}}%
      \expandafter\def\csname LT1\endcsname{\color{black}}%
      \expandafter\def\csname LT2\endcsname{\color{black}}%
      \expandafter\def\csname LT3\endcsname{\color{black}}%
      \expandafter\def\csname LT4\endcsname{\color{black}}%
      \expandafter\def\csname LT5\endcsname{\color{black}}%
      \expandafter\def\csname LT6\endcsname{\color{black}}%
      \expandafter\def\csname LT7\endcsname{\color{black}}%
      \expandafter\def\csname LT8\endcsname{\color{black}}%
    \fi
  \fi
  \setlength{\unitlength}{0.0500bp}%
  \begin{picture}(6480.00,4536.00)%
    \gplgaddtomacro\gplbacktext{%
      \csname LTb\endcsname%
      \put(1210,704){\makebox(0,0)[r]{\strut{} 0}}%
      \put(1210,1100){\makebox(0,0)[r]{\strut{} 1e-06}}%
      \put(1210,1497){\makebox(0,0)[r]{\strut{} 2e-06}}%
      \put(1210,1893){\makebox(0,0)[r]{\strut{} 3e-06}}%
      \put(1210,2289){\makebox(0,0)[r]{\strut{} 4e-06}}%
      \put(1210,2686){\makebox(0,0)[r]{\strut{} 5e-06}}%
      \put(1210,3082){\makebox(0,0)[r]{\strut{} 6e-06}}%
      \put(1210,3478){\makebox(0,0)[r]{\strut{} 7e-06}}%
      \put(1210,3875){\makebox(0,0)[r]{\strut{} 8e-06}}%
      \put(1210,4271){\makebox(0,0)[r]{\strut{} 9e-06}}%
      \put(1342,484){\makebox(0,0){\strut{} 0}}%
      \put(2290,484){\makebox(0,0){\strut{} 0.2}}%
      \put(3238,484){\makebox(0,0){\strut{} 0.4}}%
      \put(4187,484){\makebox(0,0){\strut{} 0.6}}%
      \put(5135,484){\makebox(0,0){\strut{} 0.8}}%
      \put(6083,484){\makebox(0,0){\strut{} 1}}%
      \put(176,2487){\rotatebox{-270}{\makebox(0,0){\strut{}$D(k,[(1+t)k])$}}}%
      \put(3712,154){\makebox(0,0){\strut{}$t$}}%
    }%
    \gplgaddtomacro\gplfronttext{%
    }%
    \gplbacktext
    \put(0,0){\includegraphics{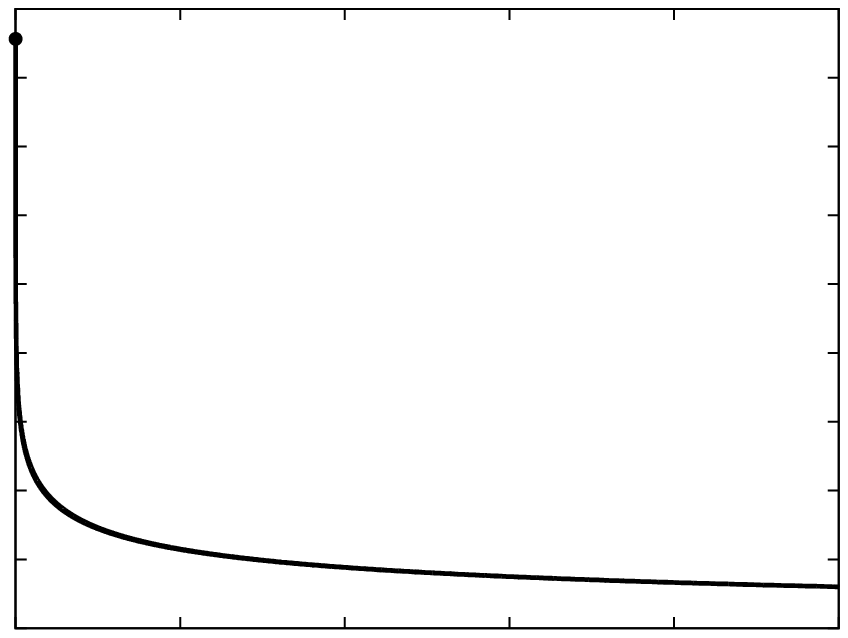}}%
    \gplfronttext
  \end{picture}%
\endgroup

%% file: figures/F.tex
\begingroup
  \makeatletter
  \providecommand\color[2][]{%
    \GenericError{(gnuplot) \space\space\space\@spaces}{%
      Package color not loaded in conjunction with
      terminal option `colourtext'%
    }{See the gnuplot documentation for explanation.%
    }{Either use 'blacktext' in gnuplot or load the package
      color.sty in LaTeX.}%
    \renewcommand\color[2][]{}%
  }%
  \providecommand\includegraphics[2][]{%
    \GenericError{(gnuplot) \space\space\space\@spaces}{%
      Package graphicx or graphics not loaded%
    }{See the gnuplot documentation for explanation.%
    }{The gnuplot epslatex terminal needs graphicx.sty or graphics.sty.}%
    \renewcommand\includegraphics[2][]{}%
  }%
  \providecommand\rotatebox[2]{#2}%
  \@ifundefined{ifGPcolor}{%
    \newif\ifGPcolor
    \GPcolortrue
  }{}%
  \@ifundefined{ifGPblacktext}{%
    \newif\ifGPblacktext
    \GPblacktexttrue
  }{}%
  \let\gplgaddtomacro\g@addto@macro
  \gdef\gplbacktext{}%
  \gdef\gplfronttext{}%
  \makeatother
  \ifGPblacktext
    \def\colorrgb#1{}%
    \def\colorgray#1{}%
  \else
    \ifGPcolor
      \def\colorrgb#1{\color[rgb]{#1}}%
      \def\colorgray#1{\color[gray]{#1}}%
      \expandafter\def\csname LTw\endcsname{\color{white}}%
      \expandafter\def\csname LTb\endcsname{\color{black}}%
      \expandafter\def\csname LTa\endcsname{\color{black}}%
      \expandafter\def\csname LT0\endcsname{\color[rgb]{1,0,0}}%
      \expandafter\def\csname LT1\endcsname{\color[rgb]{0,1,0}}%
      \expandafter\def\csname LT2\endcsname{\color[rgb]{0,0,1}}%
      \expandafter\def\csname LT3\endcsname{\color[rgb]{1,0,1}}%
      \expandafter\def\csname LT4\endcsname{\color[rgb]{0,1,1}}%
      \expandafter\def\csname LT5\endcsname{\color[rgb]{1,1,0}}%
      \expandafter\def\csname LT6\endcsname{\color[rgb]{0,0,0}}%
      \expandafter\def\csname LT7\endcsname{\color[rgb]{1,0.3,0}}%
      \expandafter\def\csname LT8\endcsname{\color[rgb]{0.5,0.5,0.5}}%
    \else
      \def\colorrgb#1{\color{black}}%
      \def\colorgray#1{\color[gray]{#1}}%
      \expandafter\def\csname LTw\endcsname{\color{white}}%
      \expandafter\def\csname LTb\endcsname{\color{black}}%
      \expandafter\def\csname LTa\endcsname{\color{black}}%
      \expandafter\def\csname LT0\endcsname{\color{black}}%
      \expandafter\def\csname LT1\endcsname{\color{black}}%
      \expandafter\def\csname LT2\endcsname{\color{black}}%
      \expandafter\def\csname LT3\endcsname{\color{black}}%
      \expandafter\def\csname LT4\endcsname{\color{black}}%
      \expandafter\def\csname LT5\endcsname{\color{black}}%
      \expandafter\def\csname LT6\endcsname{\color{black}}%
      \expandafter\def\csname LT7\endcsname{\color{black}}%
      \expandafter\def\csname LT8\endcsname{\color{black}}%
    \fi
  \fi
  \setlength{\unitlength}{0.0500bp}%
  \begin{picture}(6480.00,4536.00)%
    \gplgaddtomacro\gplbacktext{%
      \csname LTb\endcsname%
      \put(1665,704){\makebox(0,0)[r]{\strut{} 0}}%
      \put(1665,1417){\makebox(0,0)[r]{\strut{} 0.2}}%
      \put(1665,2131){\makebox(0,0)[r]{\strut{} 0.4}}%
      \put(1665,2844){\makebox(0,0)[r]{\strut{} 0.6}}%
      \put(1665,3558){\makebox(0,0)[r]{\strut{} 0.8}}%
      \put(1665,4271){\makebox(0,0)[r]{\strut{} 1}}%
      \put(1797,484){\makebox(0,0){\strut{} 0}}%
      \put(2510,484){\makebox(0,0){\strut{} 0.2}}%
      \put(3224,484){\makebox(0,0){\strut{} 0.4}}%
      \put(3937,484){\makebox(0,0){\strut{} 0.6}}%
      \put(4651,484){\makebox(0,0){\strut{} 0.8}}%
      \put(5364,484){\makebox(0,0){\strut{} 1}}%
      \put(895,2487){\rotatebox{-270}{\makebox(0,0){\strut{}$\displaystyle \sum_{s=k}^{\lfloor (1+t)k \rfloor} D(k,s)$}}}%
      \put(3580,154){\makebox(0,0){\strut{}$t$}}%
    }%
    \gplgaddtomacro\gplfronttext{%
    }%
    \gplbacktext
    \put(0,0){\includegraphics{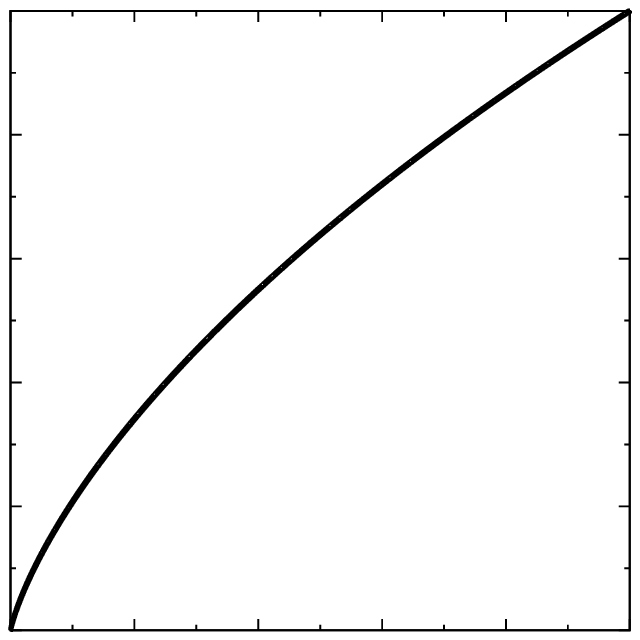}}%
    \gplfronttext
  \end{picture}%
\endgroup

%% file: figures/dFdt.tex
\begingroup
  \makeatletter
  \providecommand\color[2][]{%
    \GenericError{(gnuplot) \space\space\space\@spaces}{%
      Package color not loaded in conjunction with
      terminal option `colourtext'%
    }{See the gnuplot documentation for explanation.%
    }{Either use 'blacktext' in gnuplot or load the package
      color.sty in LaTeX.}%
    \renewcommand\color[2][]{}%
  }%
  \providecommand\includegraphics[2][]{%
    \GenericError{(gnuplot) \space\space\space\@spaces}{%
      Package graphicx or graphics not loaded%
    }{See the gnuplot documentation for explanation.%
    }{The gnuplot epslatex terminal needs graphicx.sty or graphics.sty.}%
    \renewcommand\includegraphics[2][]{}%
  }%
  \providecommand\rotatebox[2]{#2}%
  \@ifundefined{ifGPcolor}{%
    \newif\ifGPcolor
    \GPcolortrue
  }{}%
  \@ifundefined{ifGPblacktext}{%
    \newif\ifGPblacktext
    \GPblacktexttrue
  }{}%
  \let\gplgaddtomacro\g@addto@macro
  \gdef\gplbacktext{}%
  \gdef\gplfronttext{}%
  \makeatother
  \ifGPblacktext
    \def\colorrgb#1{}%
    \def\colorgray#1{}%
  \else
    \ifGPcolor
      \def\colorrgb#1{\color[rgb]{#1}}%
      \def\colorgray#1{\color[gray]{#1}}%
      \expandafter\def\csname LTw\endcsname{\color{white}}%
      \expandafter\def\csname LTb\endcsname{\color{black}}%
      \expandafter\def\csname LTa\endcsname{\color{black}}%
      \expandafter\def\csname LT0\endcsname{\color[rgb]{1,0,0}}%
      \expandafter\def\csname LT1\endcsname{\color[rgb]{0,1,0}}%
      \expandafter\def\csname LT2\endcsname{\color[rgb]{0,0,1}}%
      \expandafter\def\csname LT3\endcsname{\color[rgb]{1,0,1}}%
      \expandafter\def\csname LT4\endcsname{\color[rgb]{0,1,1}}%
      \expandafter\def\csname LT5\endcsname{\color[rgb]{1,1,0}}%
      \expandafter\def\csname LT6\endcsname{\color[rgb]{0,0,0}}%
      \expandafter\def\csname LT7\endcsname{\color[rgb]{1,0.3,0}}%
      \expandafter\def\csname LT8\endcsname{\color[rgb]{0.5,0.5,0.5}}%
    \else
      \def\colorrgb#1{\color{black}}%
      \def\colorgray#1{\color[gray]{#1}}%
      \expandafter\def\csname LTw\endcsname{\color{white}}%
      \expandafter\def\csname LTb\endcsname{\color{black}}%
      \expandafter\def\csname LTa\endcsname{\color{black}}%
      \expandafter\def\csname LT0\endcsname{\color{black}}%
      \expandafter\def\csname LT1\endcsname{\color{black}}%
      \expandafter\def\csname LT2\endcsname{\color{black}}%
      \expandafter\def\csname LT3\endcsname{\color{black}}%
      \expandafter\def\csname LT4\endcsname{\color{black}}%
      \expandafter\def\csname LT5\endcsname{\color{black}}%
      \expandafter\def\csname LT6\endcsname{\color{black}}%
      \expandafter\def\csname LT7\endcsname{\color{black}}%
      \expandafter\def\csname LT8\endcsname{\color{black}}%
    \fi
  \fi
  \setlength{\unitlength}{0.0500bp}%
  \begin{picture}(6480.00,4536.00)%
    \gplgaddtomacro\gplbacktext{%
      \csname LTb\endcsname%
      \put(682,704){\makebox(0,0)[r]{\strut{} 0}}%
      \put(682,1100){\makebox(0,0)[r]{\strut{} 1}}%
      \put(682,1497){\makebox(0,0)[r]{\strut{} 2}}%
      \put(682,1893){\makebox(0,0)[r]{\strut{} 3}}%
      \put(682,2289){\makebox(0,0)[r]{\strut{} 4}}%
      \put(682,2686){\makebox(0,0)[r]{\strut{} 5}}%
      \put(682,3082){\makebox(0,0)[r]{\strut{} 6}}%
      \put(682,3478){\makebox(0,0)[r]{\strut{} 7}}%
      \put(682,3875){\makebox(0,0)[r]{\strut{} 8}}%
      \put(682,4271){\makebox(0,0)[r]{\strut{} 9}}%
      \put(814,484){\makebox(0,0){\strut{} 0}}%
      \put(1868,484){\makebox(0,0){\strut{} 0.2}}%
      \put(2922,484){\makebox(0,0){\strut{} 0.4}}%
      \put(3975,484){\makebox(0,0){\strut{} 0.6}}%
      \put(5029,484){\makebox(0,0){\strut{} 0.8}}%
      \put(6083,484){\makebox(0,0){\strut{} 1}}%
      \put(176,2487){\rotatebox{-270}{\makebox(0,0){\strut{}$kD(k,[(1+t)k])$}}}%
      \put(3448,154){\makebox(0,0){\strut{}$t$}}%
    }%
    \gplgaddtomacro\gplfronttext{%
    }%
    \gplbacktext
    \put(0,0){\includegraphics{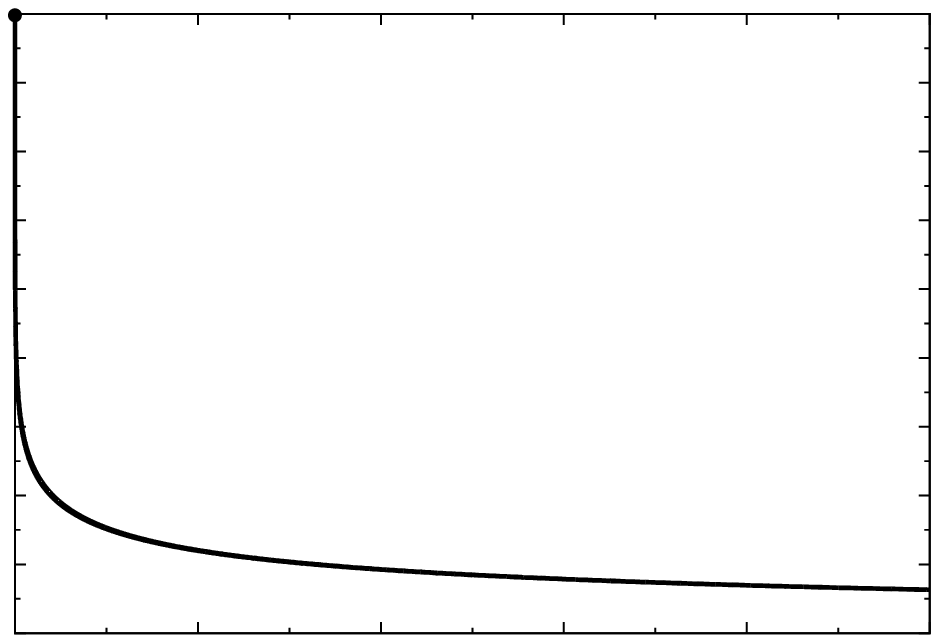}}%
    \gplfronttext
  \end{picture}%
\endgroup

%% file: figures/kDkk.tex
\begingroup
  \makeatletter
  \providecommand\color[2][]{%
    \GenericError{(gnuplot) \space\space\space\@spaces}{%
      Package color not loaded in conjunction with
      terminal option `colourtext'%
    }{See the gnuplot documentation for explanation.%
    }{Either use 'blacktext' in gnuplot or load the package
      color.sty in LaTeX.}%
    \renewcommand\color[2][]{}%
  }%
  \providecommand\includegraphics[2][]{%
    \GenericError{(gnuplot) \space\space\space\@spaces}{%
      Package graphicx or graphics not loaded%
    }{See the gnuplot documentation for explanation.%
    }{The gnuplot epslatex terminal needs graphicx.sty or graphics.sty.}%
    \renewcommand\includegraphics[2][]{}%
  }%
  \providecommand\rotatebox[2]{#2}%
  \@ifundefined{ifGPcolor}{%
    \newif\ifGPcolor
    \GPcolortrue
  }{}%
  \@ifundefined{ifGPblacktext}{%
    \newif\ifGPblacktext
    \GPblacktexttrue
  }{}%
  \let\gplgaddtomacro\g@addto@macro
  \gdef\gplbacktext{}%
  \gdef\gplfronttext{}%
  \makeatother
  \ifGPblacktext
    \def\colorrgb#1{}%
    \def\colorgray#1{}%
  \else
    \ifGPcolor
      \def\colorrgb#1{\color[rgb]{#1}}%
      \def\colorgray#1{\color[gray]{#1}}%
      \expandafter\def\csname LTw\endcsname{\color{white}}%
      \expandafter\def\csname LTb\endcsname{\color{black}}%
      \expandafter\def\csname LTa\endcsname{\color{black}}%
      \expandafter\def\csname LT0\endcsname{\color[rgb]{1,0,0}}%
      \expandafter\def\csname LT1\endcsname{\color[rgb]{0,1,0}}%
      \expandafter\def\csname LT2\endcsname{\color[rgb]{0,0,1}}%
      \expandafter\def\csname LT3\endcsname{\color[rgb]{1,0,1}}%
      \expandafter\def\csname LT4\endcsname{\color[rgb]{0,1,1}}%
      \expandafter\def\csname LT5\endcsname{\color[rgb]{1,1,0}}%
      \expandafter\def\csname LT6\endcsname{\color[rgb]{0,0,0}}%
      \expandafter\def\csname LT7\endcsname{\color[rgb]{1,0.3,0}}%
      \expandafter\def\csname LT8\endcsname{\color[rgb]{0.5,0.5,0.5}}%
    \else
      \def\colorrgb#1{\color{black}}%
      \def\colorgray#1{\color[gray]{#1}}%
      \expandafter\def\csname LTw\endcsname{\color{white}}%
      \expandafter\def\csname LTb\endcsname{\color{black}}%
      \expandafter\def\csname LTa\endcsname{\color{black}}%
      \expandafter\def\csname LT0\endcsname{\color{black}}%
      \expandafter\def\csname LT1\endcsname{\color{black}}%
      \expandafter\def\csname LT2\endcsname{\color{black}}%
      \expandafter\def\csname LT3\endcsname{\color{black}}%
      \expandafter\def\csname LT4\endcsname{\color{black}}%
      \expandafter\def\csname LT5\endcsname{\color{black}}%
      \expandafter\def\csname LT6\endcsname{\color{black}}%
      \expandafter\def\csname LT7\endcsname{\color{black}}%
      \expandafter\def\csname LT8\endcsname{\color{black}}%
    \fi
  \fi
  \setlength{\unitlength}{0.0500bp}%
  \begin{picture}(6480.00,4536.00)%
    \gplgaddtomacro\gplbacktext{%
      \csname LTb\endcsname%
      \put(814,704){\makebox(0,0)[r]{\strut{} 0}}%
      \put(814,1214){\makebox(0,0)[r]{\strut{} 2}}%
      \put(814,1723){\makebox(0,0)[r]{\strut{} 4}}%
      \put(814,2233){\makebox(0,0)[r]{\strut{} 6}}%
      \put(814,2742){\makebox(0,0)[r]{\strut{} 8}}%
      \put(814,3252){\makebox(0,0)[r]{\strut{} 10}}%
      \put(814,3761){\makebox(0,0)[r]{\strut{} 12}}%
      \put(814,4271){\makebox(0,0)[r]{\strut{} 14}}%
      \put(946,484){\makebox(0,0){\strut{}$10^{0}$}}%
      \put(1460,484){\makebox(0,0){\strut{}$10^{1}$}}%
      \put(1973,484){\makebox(0,0){\strut{}$10^{2}$}}%
      \put(2487,484){\makebox(0,0){\strut{}$10^{3}$}}%
      \put(3001,484){\makebox(0,0){\strut{}$10^{4}$}}%
      \put(3515,484){\makebox(0,0){\strut{}$10^{5}$}}%
      \put(4028,484){\makebox(0,0){\strut{}$10^{6}$}}%
      \put(4542,484){\makebox(0,0){\strut{}$10^{7}$}}%
      \put(5056,484){\makebox(0,0){\strut{}$10^{8}$}}%
      \put(5569,484){\makebox(0,0){\strut{}$10^{9}$}}%
      \put(6083,484){\makebox(0,0){\strut{}$10^{10}$}}%
      \put(176,2487){\rotatebox{-270}{\makebox(0,0){\strut{}$kD(k,k)$}}}%
      \put(3514,154){\makebox(0,0){\strut{}$k$}}%
    }%
    \gplgaddtomacro\gplfronttext{%
    }%
    \gplbacktext
    \put(0,0){\includegraphics{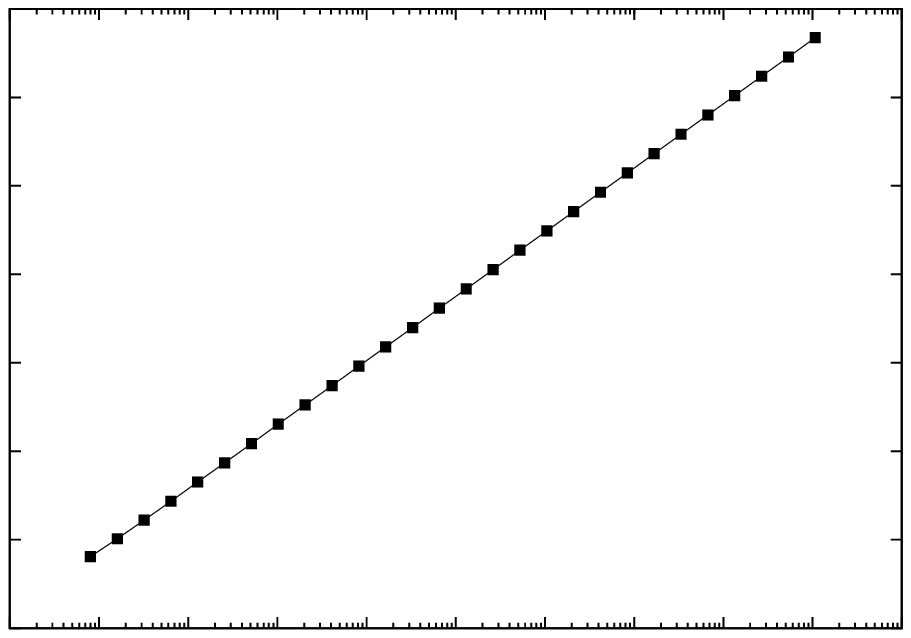}}%
    \gplfronttext
  \end{picture}%
\endgroup

%% file: figures/kDk2k.tex
\begingroup
  \makeatletter
  \providecommand\color[2][]{%
    \GenericError{(gnuplot) \space\space\space\@spaces}{%
      Package color not loaded in conjunction with
      terminal option `colourtext'%
    }{See the gnuplot documentation for explanation.%
    }{Either use 'blacktext' in gnuplot or load the package
      color.sty in LaTeX.}%
    \renewcommand\color[2][]{}%
  }%
  \providecommand\includegraphics[2][]{%
    \GenericError{(gnuplot) \space\space\space\@spaces}{%
      Package graphicx or graphics not loaded%
    }{See the gnuplot documentation for explanation.%
    }{The gnuplot epslatex terminal needs graphicx.sty or graphics.sty.}%
    \renewcommand\includegraphics[2][]{}%
  }%
  \providecommand\rotatebox[2]{#2}%
  \@ifundefined{ifGPcolor}{%
    \newif\ifGPcolor
    \GPcolortrue
  }{}%
  \@ifundefined{ifGPblacktext}{%
    \newif\ifGPblacktext
    \GPblacktexttrue
  }{}%
  \let\gplgaddtomacro\g@addto@macro
  \gdef\gplbacktext{}%
  \gdef\gplfronttext{}%
  \makeatother
  \ifGPblacktext
    \def\colorrgb#1{}%
    \def\colorgray#1{}%
  \else
    \ifGPcolor
      \def\colorrgb#1{\color[rgb]{#1}}%
      \def\colorgray#1{\color[gray]{#1}}%
      \expandafter\def\csname LTw\endcsname{\color{white}}%
      \expandafter\def\csname LTb\endcsname{\color{black}}%
      \expandafter\def\csname LTa\endcsname{\color{black}}%
      \expandafter\def\csname LT0\endcsname{\color[rgb]{1,0,0}}%
      \expandafter\def\csname LT1\endcsname{\color[rgb]{0,1,0}}%
      \expandafter\def\csname LT2\endcsname{\color[rgb]{0,0,1}}%
      \expandafter\def\csname LT3\endcsname{\color[rgb]{1,0,1}}%
      \expandafter\def\csname LT4\endcsname{\color[rgb]{0,1,1}}%
      \expandafter\def\csname LT5\endcsname{\color[rgb]{1,1,0}}%
      \expandafter\def\csname LT6\endcsname{\color[rgb]{0,0,0}}%
      \expandafter\def\csname LT7\endcsname{\color[rgb]{1,0.3,0}}%
      \expandafter\def\csname LT8\endcsname{\color[rgb]{0.5,0.5,0.5}}%
    \else
      \def\colorrgb#1{\color{black}}%
      \def\colorgray#1{\color[gray]{#1}}%
      \expandafter\def\csname LTw\endcsname{\color{white}}%
      \expandafter\def\csname LTb\endcsname{\color{black}}%
      \expandafter\def\csname LTa\endcsname{\color{black}}%
      \expandafter\def\csname LT0\endcsname{\color{black}}%
      \expandafter\def\csname LT1\endcsname{\color{black}}%
      \expandafter\def\csname LT2\endcsname{\color{black}}%
      \expandafter\def\csname LT3\endcsname{\color{black}}%
      \expandafter\def\csname LT4\endcsname{\color{black}}%
      \expandafter\def\csname LT5\endcsname{\color{black}}%
      \expandafter\def\csname LT6\endcsname{\color{black}}%
      \expandafter\def\csname LT7\endcsname{\color{black}}%
      \expandafter\def\csname LT8\endcsname{\color{black}}%
    \fi
  \fi
  \setlength{\unitlength}{0.0500bp}%
  \begin{picture}(6480.00,4536.00)%
    \gplgaddtomacro\gplbacktext{%
      \csname LTb\endcsname%
      \put(1210,704){\makebox(0,0)[r]{\strut{} 0.59}}%
      \put(1210,1100){\makebox(0,0)[r]{\strut{} 0.595}}%
      \put(1210,1497){\makebox(0,0)[r]{\strut{} 0.6}}%
      \put(1210,1893){\makebox(0,0)[r]{\strut{} 0.605}}%
      \put(1210,2289){\makebox(0,0)[r]{\strut{} 0.61}}%
      \put(1210,2686){\makebox(0,0)[r]{\strut{} 0.615}}%
      \put(1210,3082){\makebox(0,0)[r]{\strut{} 0.62}}%
      \put(1210,3478){\makebox(0,0)[r]{\strut{} 0.625}}%
      \put(1210,3875){\makebox(0,0)[r]{\strut{} 0.63}}%
      \put(1210,4271){\makebox(0,0)[r]{\strut{} 0.635}}%
      \put(1342,484){\makebox(0,0){\strut{}$10^{0}$}}%
      \put(1816,484){\makebox(0,0){\strut{}$10^{1}$}}%
      \put(2290,484){\makebox(0,0){\strut{}$10^{2}$}}%
      \put(2764,484){\makebox(0,0){\strut{}$10^{3}$}}%
      \put(3238,484){\makebox(0,0){\strut{}$10^{4}$}}%
      \put(3713,484){\makebox(0,0){\strut{}$10^{5}$}}%
      \put(4187,484){\makebox(0,0){\strut{}$10^{6}$}}%
      \put(4661,484){\makebox(0,0){\strut{}$10^{7}$}}%
      \put(5135,484){\makebox(0,0){\strut{}$10^{8}$}}%
      \put(5609,484){\makebox(0,0){\strut{}$10^{9}$}}%
      \put(6083,484){\makebox(0,0){\strut{}$10^{10}$}}%
      \put(176,2487){\rotatebox{-270}{\makebox(0,0){\strut{}$kD(k,2k)$}}}%
      \put(3712,154){\makebox(0,0){\strut{}$k$}}%
    }%
    \gplgaddtomacro\gplfronttext{%
    }%
    \gplbacktext
    \put(0,0){\includegraphics{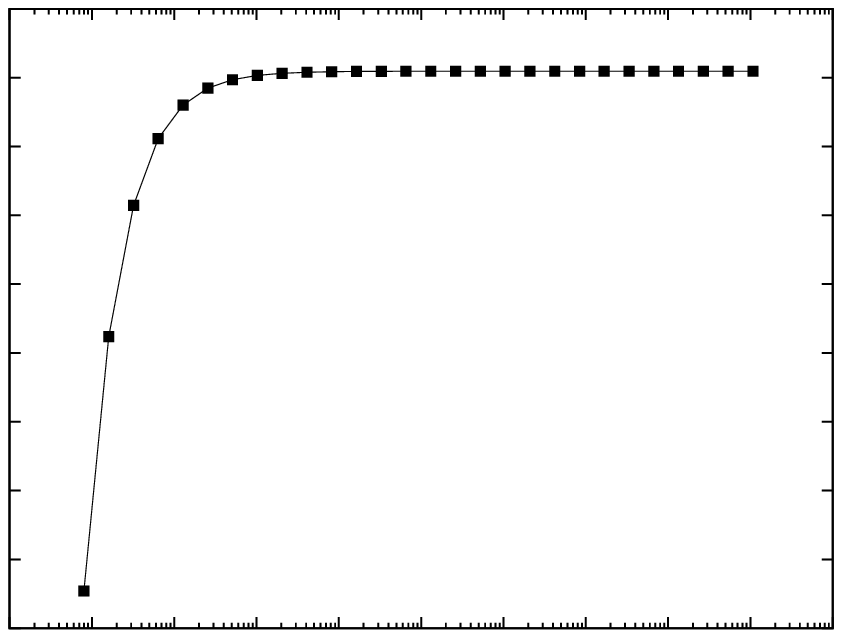}}%
    \gplfronttext
  \end{picture}%
\endgroup

%% file: figures/D.tex
\begingroup
  \makeatletter
  \providecommand\color[2][]{%
    \GenericError{(gnuplot) \space\space\space\@spaces}{%
      Package color not loaded in conjunction with
      terminal option `colourtext'%
    }{See the gnuplot documentation for explanation.%
    }{Either use 'blacktext' in gnuplot or load the package
      color.sty in LaTeX.}%
    \renewcommand\color[2][]{}%
  }%
  \providecommand\includegraphics[2][]{%
    \GenericError{(gnuplot) \space\space\space\@spaces}{%
      Package graphicx or graphics not loaded%
    }{See the gnuplot documentation for explanation.%
    }{The gnuplot epslatex terminal needs graphicx.sty or graphics.sty.}%
    \renewcommand\includegraphics[2][]{}%
  }%
  \providecommand\rotatebox[2]{#2}%
  \@ifundefined{ifGPcolor}{%
    \newif\ifGPcolor
    \GPcolortrue
  }{}%
  \@ifundefined{ifGPblacktext}{%
    \newif\ifGPblacktext
    \GPblacktexttrue
  }{}%
  \let\gplgaddtomacro\g@addto@macro
  \gdef\gplbacktext{}%
  \gdef\gplfronttext{}%
  \makeatother
  \ifGPblacktext
    \def\colorrgb#1{}%
    \def\colorgray#1{}%
  \else
    \ifGPcolor
      \def\colorrgb#1{\color[rgb]{#1}}%
      \def\colorgray#1{\color[gray]{#1}}%
      \expandafter\def\csname LTw\endcsname{\color{white}}%
      \expandafter\def\csname LTb\endcsname{\color{black}}%
      \expandafter\def\csname LTa\endcsname{\color{black}}%
      \expandafter\def\csname LT0\endcsname{\color[rgb]{1,0,0}}%
      \expandafter\def\csname LT1\endcsname{\color[rgb]{0,1,0}}%
      \expandafter\def\csname LT2\endcsname{\color[rgb]{0,0,1}}%
      \expandafter\def\csname LT3\endcsname{\color[rgb]{1,0,1}}%
      \expandafter\def\csname LT4\endcsname{\color[rgb]{0,1,1}}%
      \expandafter\def\csname LT5\endcsname{\color[rgb]{1,1,0}}%
      \expandafter\def\csname LT6\endcsname{\color[rgb]{0,0,0}}%
      \expandafter\def\csname LT7\endcsname{\color[rgb]{1,0.3,0}}%
      \expandafter\def\csname LT8\endcsname{\color[rgb]{0.5,0.5,0.5}}%
    \else
      \def\colorrgb#1{\color{black}}%
      \def\colorgray#1{\color[gray]{#1}}%
      \expandafter\def\csname LTw\endcsname{\color{white}}%
      \expandafter\def\csname LTb\endcsname{\color{black}}%
      \expandafter\def\csname LTa\endcsname{\color{black}}%
      \expandafter\def\csname LT0\endcsname{\color{black}}%
      \expandafter\def\csname LT1\endcsname{\color{black}}%
      \expandafter\def\csname LT2\endcsname{\color{black}}%
      \expandafter\def\csname LT3\endcsname{\color{black}}%
      \expandafter\def\csname LT4\endcsname{\color{black}}%
      \expandafter\def\csname LT5\endcsname{\color{black}}%
      \expandafter\def\csname LT6\endcsname{\color{black}}%
      \expandafter\def\csname LT7\endcsname{\color{black}}%
      \expandafter\def\csname LT8\endcsname{\color{black}}%
    \fi
  \fi
  \setlength{\unitlength}{0.0500bp}%
  \begin{picture}(7200.00,5040.00)%
    \gplgaddtomacro\gplbacktext{%
      \csname LTb\endcsname%
      \put(946,704){\makebox(0,0)[r]{\strut{}$10^{-7}$}}%
      \put(946,1383){\makebox(0,0)[r]{\strut{}$10^{-6}$}}%
      \put(946,2061){\makebox(0,0)[r]{\strut{}$10^{-5}$}}%
      \put(946,2740){\makebox(0,0)[r]{\strut{}$10^{-4}$}}%
      \put(946,3418){\makebox(0,0)[r]{\strut{}$10^{-3}$}}%
      \put(946,4097){\makebox(0,0)[r]{\strut{}$10^{-2}$}}%
      \put(946,4775){\makebox(0,0)[r]{\strut{}$10^{-1}$}}%
      \put(1078,484){\makebox(0,0){\strut{}$10^{0}$}}%
      \put(1896,484){\makebox(0,0){\strut{}$10^{1}$}}%
      \put(2714,484){\makebox(0,0){\strut{}$10^{2}$}}%
      \put(3532,484){\makebox(0,0){\strut{}$10^{3}$}}%
      \put(4349,484){\makebox(0,0){\strut{}$10^{4}$}}%
      \put(5167,484){\makebox(0,0){\strut{}$10^{5}$}}%
      \put(5985,484){\makebox(0,0){\strut{}$10^{6}$}}%
      \put(6803,484){\makebox(0,0){\strut{}$10^{7}$}}%
      \put(176,2739){\rotatebox{-270}{\makebox(0,0){\strut{}$D(k)-m$}}}%
      \put(3940,154){\makebox(0,0){\strut{}$k$}}%
    }%
    \gplgaddtomacro\gplfronttext{%
    }%
    \gplbacktext
    \put(0,0){\includegraphics{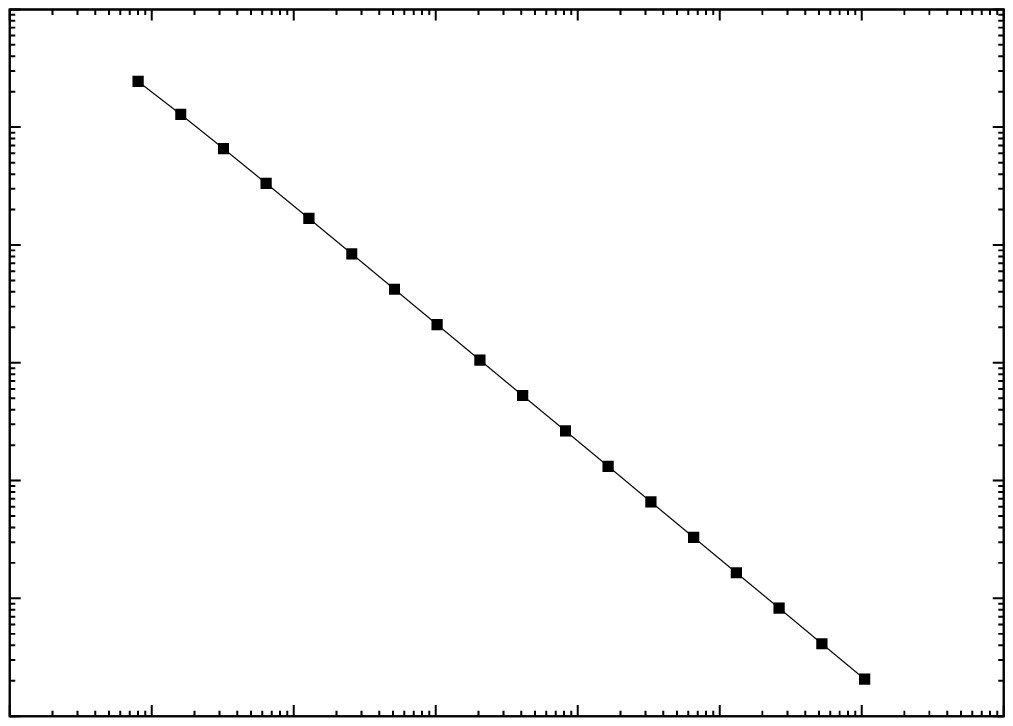}}%
    \gplfronttext
  \end{picture}%
\endgroup